\documentclass[12pt]{amsart}
\usepackage{amscd,amssymb}%amsmath}
\usepackage{graphicx,enumerate}
\usepackage{upref}
\usepackage{amsmath}
\usepackage{amstext}
\usepackage{amsthm}

\usepackage{color}
\usepackage[arrow,matrix,ps,color,line,curve,frame]{xy}
\usepackage{pgf, tikz}
\usepackage{url}
\usepackage{enumerate}

\newcommand\encircle[1]{%
  \tikz[baseline=(X.base)] 
    \node (X) [draw, shape=circle, inner sep=0] {\strut #1};}

\usetikzlibrary{matrix,arrows}

\usepackage[colorinlistoftodos,bordercolor=red,backgroundcolor=red!20,linecolor=red,textsize=scriptsize]{todonotes}

\let\oldlabel=\label

\def\Drellabel{\marginparsep=1em
    \def\label##1{\oldlabel{##1}\ifmmode\else\ifinner\else
         \marginpar{{\footnotesize\ \\ \tt
                    ##1}}\fi\fi}}
\def\inte{\operatorname{int}}
\def\conv{\operatorname{conv}}
\def\vertex{\operatorname{vert}}
\def\Im{\operatorname{Im}}
\def\Aff{\operatorname{Aff}}
\def\B{\operatorname{B}}

\def\D{\underset\Delta}
\def\d{{\text{d}_{\text H}}}
\def\q{\underset{\bf q}}
\def\Infty{\underset{\infty}}

\def\R{{\bf R}}
\def\SS{{\bf S}}
\def\RR{{\mathbb R}}
\def\QQ{{\mathbb Q}}
\def\ZZ{{\mathbb Z}}
\def\NN{{\mathbb N}}
\def\FF{{\mathbb F}}

\def\POL{\mathbf{Pol}}
\def\np{{\operatorname{NPol}}}
\def\cones{\operatorname{Cones}}

\def\viii{{\text{\tiny\encircle{i}}}}

\def\vii{{\text{\tiny\encircle{ii}}}}

\def\v{\text{\tiny\encircle{iii}}}

\def\vi{\text{\tiny\encircle{iv}}}

\def\ix{\text{\tiny\encircle{v}}}

\let\phi=\varphi
\let\theta=\vartheta
\let\epsilon=\varepsilon

\advance\headheight1.15pt

\newtheorem{lemma}{Lemma}[section]
\newtheorem{corollary}[lemma]{Corollary}

\newtheorem{proposition}[lemma]{Proposition}

\theoremstyle{definition}
\newtheorem{definition}[lemma]{Definition}
\newtheorem{remark}[lemma]{Remark}

\newtheorem{innercustomthm}{Theorem}
\newenvironment{customthm}[1]
  {\renewcommand\theinnercustomthm{#1}\innercustomthm}
  {\endinnercustomthm}

\begin{document}

\title[The pyramidal growth]{The pyramidal growth}

\begin{abstract}
Can one build an arbitrary polytope from any polytope inside by iteratively stacking pyramids onto facets, without losing the convexity throughout the process? We prove that this is indeed possible for (i) 3-polytopes, (ii) 4-polytopes under a certain infinitesimal quasi-pyramidal relaxation, and (iii) all dimensions asymptotically. The motivation partly comes from our study of $K$-theory of monoid rings and of certain posets of discrete-convex objects.
%\medskip\noindent Peut-on construire un polytope arbitraire \`a partir de n'importe quel polytope \`a l'int\'erieur en empilant it\'erativement des pyramides sur des facettes, sans perdre la convexit\'e tout au long du processus? Nous prouvons que cela est en effet possible pour (i) 3-polytopes, (ii) 4-polytopes sous une certaine relaxation quasi-pyramidale infinit\'esimale, et (iii) toutes les dimensions asymptotiquement. La motivation vient en partie de notre étude du K-th\'eorie des anneaux monoïdes et de certains posets d'objets discrets-convexes. 
\end{abstract}

\author{Joseph Gubeladze}

\address{Department of Mathematics\\
         San Francisco State University\\
         1600 Holloway Ave.\\
         San Francisco, CA 94132, USA}
\email{soso@sfsu.edu}

\thanks{}



\maketitle

\section{Main results}\label{Main}

A \emph{polytope} in this paper means the convex hull of a finite subset of $\oplus_\NN\RR$, where we have the usual notion of convexity, Euclidean norm, angle between two finite dimensional affine spaces that meet in codimension one, topological closure etc. 

The topological closure of a subset $X\subset\oplus_\NN\RR$ will be denoted by $\overline X$. 

The \emph{Hausdrorff distance} between two nonempty compact subsets $X,Y\subset\oplus_\NN\RR$ will be denoted by $\d(X,Y)$ \cite[Ch.~1.2]{Grunbaum}. For a sequence of polytopes $\{P_i\}_\NN$ and a polytope $Q$, we write $\underset{i\to\infty}\lim P_i=Q$ if $\underset{i\to\infty}\lim\d(P_i,Q)=0$.

The set of polytopes of dimension at most $d$ will be denoted by $\POL(d)$. The set of all polytopes will be denoted by $\POL(\infty)$. For a subfield $k\subset\RR$, the corresponding sets of polytopes with vertices in $\oplus_\NN k$ will be denoted by $\POL_k(d)$ and $\POL_k(\infty)$. 

Let $P$ be a polytope. A \emph{pyramid over} or \emph{with base} $P$ is the convex hull $Q$ of $P$ and a point $v$, not in the affine hull of $P$. The point $v$ is the \emph{apex} of $Q$.

Let $k\subset\RR$ be a subfield, $d\le\infty$, and $\POL$ be one of the sets $\POL(d)$, $\POL_k(d)$.

\begin{definition}\label{Pyrmidal_growth}
(a) A pair of polytopes $P\subset Q$ in $\POL$ forms a \emph{pyramidal extension} if $\Delta=\overline{Q\setminus P}$ is a pyramid and $\Delta\cap P$ is a facet of $\Delta$, i.e., either $Q$ is a pyramid over $P$ or obtained from $P$ by stacking a pyramid onto a facet. For a pyramidal extension $P\subset Q$ we write $P\D\subset Q$.

\noindent(b) The partial order on $\POL$, generated by the pyramidal extensions within $\POL$, will be denoted by $\D\le$ and called the \emph{pyramidal growth}.

\noindent(c) The \emph{transfinite pyramidal growth} $\Infty\le$ is the smallest partial order  on $\POL$, containing $\D\le$ and satisfying $P\Infty\le Q$ whenever there exists an ascending sequence  $P=P_0\Infty\le P_1\Infty\le P_2\Infty\le\ldots$ with  $Q=\underset{i\to\infty}\lim P_i$. 
\end{definition}

It is easily shown that $\D\le$ is the inclusion order on $\POL(2)$; see Corollary \ref{polygons}. In Figure \ref{2growth_fig}, the outer quadrilateral is grown from the hexagon inside by stacking triangles onto edges in the indicated order:
\begin{figure}[h!]
\caption{}
\vspace{.1in}
\includegraphics[%trim = 0mm 0.1in 0mm 0.1in, clip, 
scale=1.5]{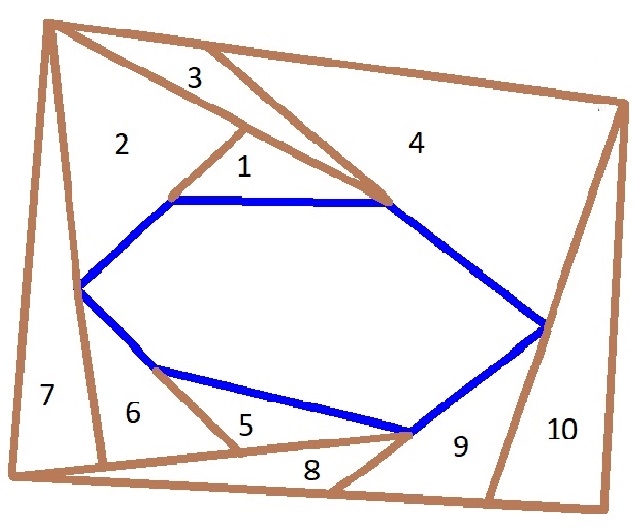}\label{2growth_fig}
\end{figure}

Pyramidal extensions are more general than the extensions, used in the definition of \emph{stacked polytopes} \cite[Ch.~3.19]{Brondsted}\cite[Ch.~10.6]{Grunbaum}, along with their direct generalization to arbitrary initial polytopes: by only allowing the stackings of pyramids onto facets when none of the codimension 2 faces disappears we get a new partial order on polytopes. Obviously, it does not coincide with the inclusion order. We do not know whether the transfinite completion of this order is the same as $\Infty\le$.

\begin{definition}\label{Quasi_growth}
(a) A \emph{quasi-pyramidal growth} in $\POL$, is a pair of polytopes $P\subset Q$, admitting within $\POL$ a finite sequence of polytopes
\begin{align*}
P=P_0\subset P_1\subset\ldots\subset P_n=Q
\end{align*}
and pyramidal extensions
$$
P'_i\D\subset P_i,\qquad i=1,\ldots,n,
$$
such that 
$$
P_0\subset P_i'\subset P_{i-1},\qquad i=1,\ldots,n.
$$
The resulting partial order on $\POL$ will be denoted by $\q\le$. 

\noindent(b) The \emph{quasi-pyramidal defect} of a pair of polytopes $P\q\le Q$ in $\POL$ is defined by
\begin{align*}
\delta(P,Q)=\inf\bigg\{\sum \d(P_i',P_i)\ \bigg |\ P'_i\subsetneq P_{i-1}\bigg\}, 
\end{align*}
where the infimum is taken over the sequences as in the part (a).
\end{definition}

Informally, the quasi-pyramidal defect measures how close an equality $P\q\le Q$ is to the inequality $P\D\le Q$. Observe that the first extension $P_0\subset P_1$ in Definition \ref{Quasi_growth}(a) is necessarily pyramidal.

\medskip Recently several notions of minimal enlargements of polytopes appeared in the literature: extensions of lattice 3-polytopes by adding one lattice point play an important role in a classification of lattice 3-polytopes \cite{Paco}; connectivity of the graph on the set of polytopes is studied in \cite{Elementary_moves}, where two polytopes form an edge  if their vertex sets differ by adding/deleting one element.

Let $k\subset\RR$ be a subfield. Our main results are:

\begin{customthm}{A}\label{Theorem_A}
\emph{$\Infty\le$ is the inclusion order on $\POL_k(\infty)$.} 
\end{customthm}

\begin{customthm}{B}\label{Theorem_B}
\emph{$\D\le$ is the inclusion order on $\POL_k(3)$.}
\end{customthm}

The classical proof of the \emph{Steinitz Theorem} \cite[Ch.~13.1]{Grunbaum} implies that the set of 3-polytopes in $\RR^3$ is connected via combinatorial  modifications, which can be realized geometrically as `up-down' pyramidal extensions, when the pyramids being stacked are simplices. On the other hand, Theorem \ref{Theorem_B} implies that the poset of 3-polytopes in $\RR^3$, ordered by $\D\le$, is topologically contractible, i.e., the geometric realization of the corresponding order complex is a contractible space: every finite system in this poset has an upper bound%\todo{fixed}.

\begin{customthm}{C}\label{Theorem_C}
\emph{If $\D\le$ is the inclusion order on $\POL_k(d)$ then $\q\le$ is the inclusion order on $\POL_k(d+1)$. Moreover, $\delta(P,Q)=0$ for any two polytopes $P\subset Q$ in $\POL_k(d+1)$.}
\end{customthm}

\medskip In particular, $\q\le$ is the inclusion order on $\POL_k(4)$ and $\delta(P,Q)=0$ for any two polytopes $P\subset Q$ in $\POL_k(4)$. The proof of Theorem \ref{Theorem_C} also implies that the same result on quasi-pyramidal growth can be proved unconditionally in all dimensions if a local conical version of the induction step can be worked out; see Remark \ref{Almost_quasi}.

\bigskip\noindent{\bf\emph{$K$-theory.}} The poset $\big(\POL_\QQ(\infty),\D\le\big)$ is implicit in our $K$-theoretic works on monoid rings. Informally, a pyramidal extension $P\D\subset Q$ represents a minimal enlargement of a polytope, allowing to transfer certain $K$-theoretic information on $P$ to the polytope $Q$. More precisely, rational polytopes give rise to  submonoids of $\ZZ^d$ and, when $P\D\subset Q$, certain $K$-theoretic objects over the monoid ring, associated with $Q$, are extended from the submonoid ring, associated with $P$. Results of this type in various $K$-theoretic scenarios are obtained in %\cite{Anderson,Classical,Nilpotence,Steinberg,Elrows}. 
\cite{Anderson,Nilpotence,Elrows}.
Assume $\D\le$ coincides with $\subset$. Then the $K$-theoretic objects in question, defined over the monoid ring of $Q$, extend from polynomial rings because there is a rational simplex $P\subset Q$, defining a free commutative monoid. But $K$-theory of the monoid rings of free commutative monoid, which is the same as polynomial rings, is one of the best understood topics in algebraic $K$-theory -- the so-called \emph{homotopy invariance} properties. Most likely, the relations $\subset$ and $\D\le$ are different; see below. In the mentioned works we used the following substitute, which suffices for the $K$-theoretic purposes: for two rational polytopes $P\subset Q$, there is a sequence of rational polytopes of the form
\begin{align*}
P=&P_0,\ P_1,\ \ldots,\ P_n=Q\\
&P_i\subset Q\ \ \text{and}\ \  P_{i-1}\D\subset P_i\ \ \text{or}\ \ P_{i-1}\supset P_i,\\
&\qquad\qquad\qquad\qquad\qquad\qquad i=1,\ldots,n.
\end{align*}

With a small additional $K$-theoretic work, one can show that the relevant $K$-theoretic information still can be transferred from $P$ to $Q$ if $P\q\le Q$ or $P\Infty\le Q$. Thus Theorem \ref{Theorem_A} allows to avoid the non-monotonicity fluctuations $P_{i-1}\supset P_i$ above in the process of descending from the larger polytope $Q$ to $P$%\todo{fixed}.

\medskip\noindent{\bf\emph{Quantum jumps and rational cones.}}
In \cite{Quantum,Cones} we explored two posets: (i) the poset $\np(d)$ of \emph{normal polytopes} -- essentially the projectively normal embeddings of toric varieties -- whose minimal elements have played crucial role in disproving various covering conjectures in the 1990s \cite{Unico,Caratex}, and (ii) the poset $\cones(d)$ of \emph{rational cones}, as the additive counterpart of $\np(d-1)$ via the correspondence $P\mapsto$\ the homogenization cone of $P$. The poset of cones is more amenable to arithmetic and topological analysis and can provide a handle on the poset of normal polytopes. The elementary relation in $\np(d)$, called \emph{quantum jumps,} are the extensions of normal polytopes by adding one lattice point. The order in $\cones(d)$ is generated by the extensions of the monoids of the form $C\cap\ZZ^d$, $C\subset\RR^d$ a rational cone, by adding one generator. Currently, even the existence of isolated points in $\np(d)$ is not excluded for $d\ge4$, whereas the order in $\cones(d)$ is conjectured to be the inclusion order for any $d$. The poset $\big(\POL(d),\q\le\big)$ is a continuous analogue of $\cones(d+1)$ and so it too is expected to be ordered by inclusion. This is partially confirmed by Theorem \ref{Theorem_C}. The poset $\big(\POL(d),\D\le\big)$ is a continuous version of another poset, also introduced in \cite[\S2B]{Cones}. It is generated by the so-called \emph{height 1 extensions over facets} of cones, has fewer relations than $\cones(d+1)$, and allows extensive computational experimentation.

\medskip\noindent{\bf\emph{Do the pyramidal growth and inclusion order coincide?}} If $P\D\subset Q$ and $P$ is a rational polytope, then $Q$ is combinatorially equivalent to a rational polytope. Whether the same can be said when $P\D\le Q$ is an interesting question. In view of the existence of polytopes of irrational type (e.g., \cite[Ch.~5.4]{Grunbaum}), the positive answer would imply that $\D\le$ and $\subset$ are different in the corresponding dimension over the corresponding subfield $k\subset\RR$

Another question of independent interest, in the spirit of \emph{approximations by polytopes} \cite{Gruber}, is whether there is a sequence $0=P_0\ \D\le\ P_1\ \D\le\ \ldots$ in $\POL(d)$, starting with the origin $0\in\RR^d$ and such that $\bigcup_{i=0}^\infty P_i=\overset{\tiny\circ}\B_d$,
where $\overset{\tiny\circ}\B_d$ is the open unit $d$-ball. The negative answer would show that $\D\le$ is not the inclusion order on $\POL_\QQ(d)$.

\begin{remark}\label{Fields} (a) In the rest of the paper we only consider polytopes in $\POL(d)$. But adjusting the arguments to $\POL_k(d)$ for any subfield $k\subset\RR$ is straightforward.  

\medskip\noindent(b) The proof of the main results yields an algorithm for building up arbitrary 3-polytopes via pyramidal growth and arbitrary 4-polytopes via quasi-pyramidal growth, arbitrarily close to pyramidal growth.
\end{remark}

\section{Notation}\label{Notation}

Our references for basic facts on polytopes are \cite[Ch.1]{Kripo} and \cite{Grunbaum}. The relatively standard notation/terminology we will use is as follows.

A \emph{space} refers to a finite-dimensional affine subspace of $\oplus_\NN\RR$. A \emph{half-space} of a space will mean a \emph{closed} affine half-space.

The \emph{convex} and \emph{affine hulls} of a subset $X\subset\oplus_\NN\RR$ will be denoted by $\conv(X)$ and $\Aff(X)$, respectively. For two points $x,y\in\oplus_\NN\RR$ we will use $[x,y]$ for $\conv(x,y)$.

The \emph{conical hull} of a subset $X\subset\oplus_\NN\RR$ will be denoted by $\RR_+X$, i.e., $\RR_+X=\{cx\ |\ x\in X,\ c\ge0\}$.

For a finite dimensional convex set $X$, by $\inte(X)$ we denote the \emph{relative interior} of $X$. The \emph{boundary} of $X$ is $\partial X=X\setminus\inte(X)$. 

For a polytope $P$, the set of its \emph{vertices} and \emph{facets} will be denoted by $\vertex(P)$ and $\FF(P)$, respectively.

\medskip We also need a more specialized notation.

\medskip\noindent(i) Let $P,Q$ be polytopes, such that $f=P\cap R$ is a common face and, moreover, $\dim(\conv(P,Q))=\dim P+\dim Q-\dim f$. Then $\conv(P\cup Q)$ is the \emph{colimit} in the category of convex polytopes and affine maps of the diagram of face embeddings $P_{\nwarrow\underset f\ \nearrow}Q$; for the categorial analysis of polytopes see \cite{Functors}. Correspondingly, instead of $\conv(P,Q)$ we will use the more informative notation  $P\underset f\vee Q$. Figure \ref{PVQ_fig} represents the construction for two pentagons, sharing an edge: 
\begin{figure}[h!]
\caption{}
\vspace{.1in}
\includegraphics[%trim = 0mm 0.1in 0mm 0.1in, clip, 
scale=1]{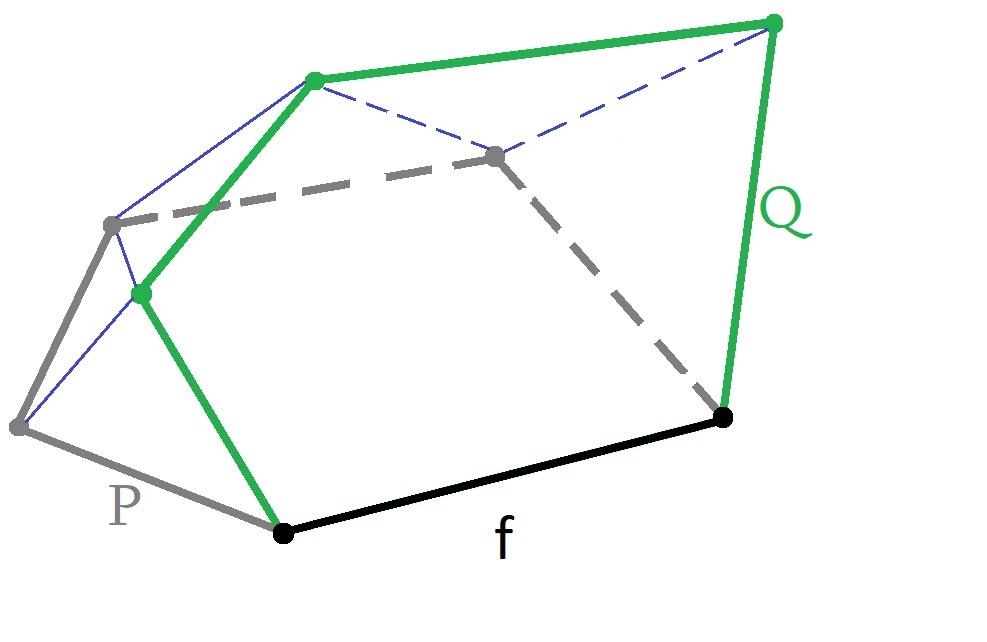}\label{PVQ_fig}
\end{figure}

\medskip\noindent(ii) For a full-dimensional polytope $P$ in a space $H$, a point $v\in H\setminus P$, and a not necessarily proper face $P'\subset P$, we denote by $\partial_v(P')^+$ and $\FF_v(P')^+$ the part of $\partial(P')$ and the set of facets of $P'$, respectively, whose visibility is not obstructed by $P$. We skip $P$ from the notation because the polytope $P$ will be clear from the context. In Figure \ref{VISIBLE}, the set $\FF_v(P)^+$ consists of two triangles and, for the facet $P'\subset P$, the set $\FF_v(P')^+$ consists of two edges:
\begin{figure}[h!]
\caption{}
\vspace{.1in}
\includegraphics[%trim = 0mm 0.1in 0mm 0.1in, clip, 
scale=1]{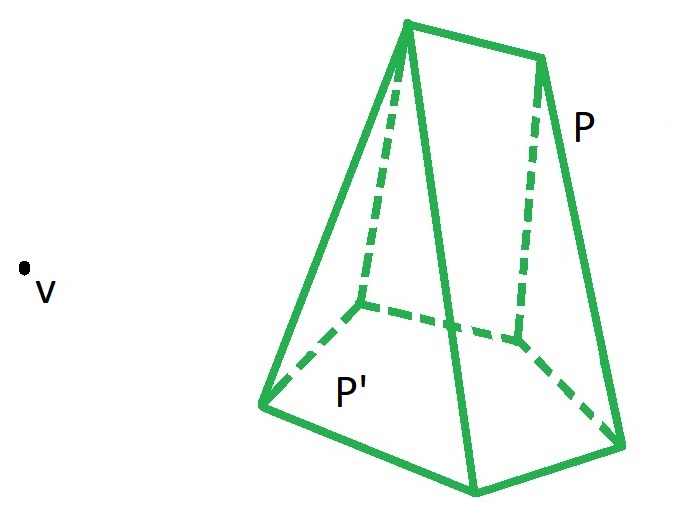}\label{VISIBLE}
\end{figure}

\medskip\noindent(iii) For a polytope $P$ and a facet $f\subset P$, we denote by $\Aff_f(P)^+\subset\Aff(P)$ the half-space, bounded by $\Aff(f)$ and containing $P$ (Figure \ref{half_fig}).
\begin{figure}[h!]
\caption{}
\vspace{.1in}
\includegraphics[%trim = 0mm 0.1in 0mm 0.1in, clip, 
scale=1.5]{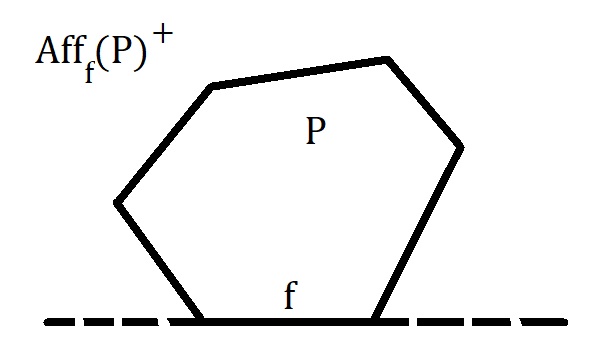}\label{half_fig}
\end{figure}

%\medskip
\noindent(iv) For a space $H$ and a point $w\notin H$, we denote by $H_w^-$ the half-space in $\Aff(H,w)$, not containing $w$ (Figure \ref{half_neg_fig}).
\begin{figure}[h!]
\caption{}
\vspace{.1in}
\includegraphics[%trim = 0mm 0.1in 0mm 0.1in, clip, 
scale=1.4]{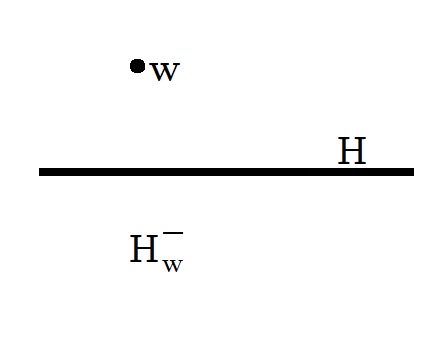}\label{half_neg_fig}
\end{figure}

\medskip\noindent(v) Consider a finite dimensional convex set $C$, a space $H$, a polytope $\pi$, and a point $z$, satisfying the conditions:

\medskip\begin{enumerate}[\rm$\centerdot$]
\item  $\pi\subset H$ and $z\in C\setminus H$,
\item $\dim(\Aff(C,H))=\dim(H)+1=\dim(\pi)+2$,
\item $C\cap H$ is contained in exactly one of the half-spaces of $H$, bounded by $\Aff(\pi)$.
\end{enumerate}
\medskip\noindent Then, for a real number $\theta\ge0$, denote by:

\medskip\begin{enumerate}[\rm$\centerdot$]
\item $H_\theta(\pi,C,z)$ the rotation of $H$ inside $\Aff(C,H)$ around $\Aff(\pi)$ by the angle $\theta$, moving the half-pace of $H$, which contains $H\cap C$, towards $z$ in such a way that $H\cap C$ stays nonempty during the rotation;
%\item $\theta_z$ the smallest positive real number, such that $z\in H_{\theta_z}(p,C,z)$;%\todo{needed1?}
\item $H_\theta(\pi,C,z)^-$ the half-space of $\Aff(C,H)$, bounded by $H_\theta(\pi,C,z)$ and not containing $z$; this half-space  exists for every sufficiently small $\theta\ge0$; this notation is a simplification of $\big(H_\theta(\pi,C,z)\big)_z^-$, which results from (iv) above.

\medskip Figure \ref{rotation_fig} represents the notation introduced:

\medskip
\begin{figure}[h!]
\caption{}
\vspace{.25in}
\includegraphics[%trim = 0mm 0.1in 0mm 0.1in, clip, 
scale=1.2]{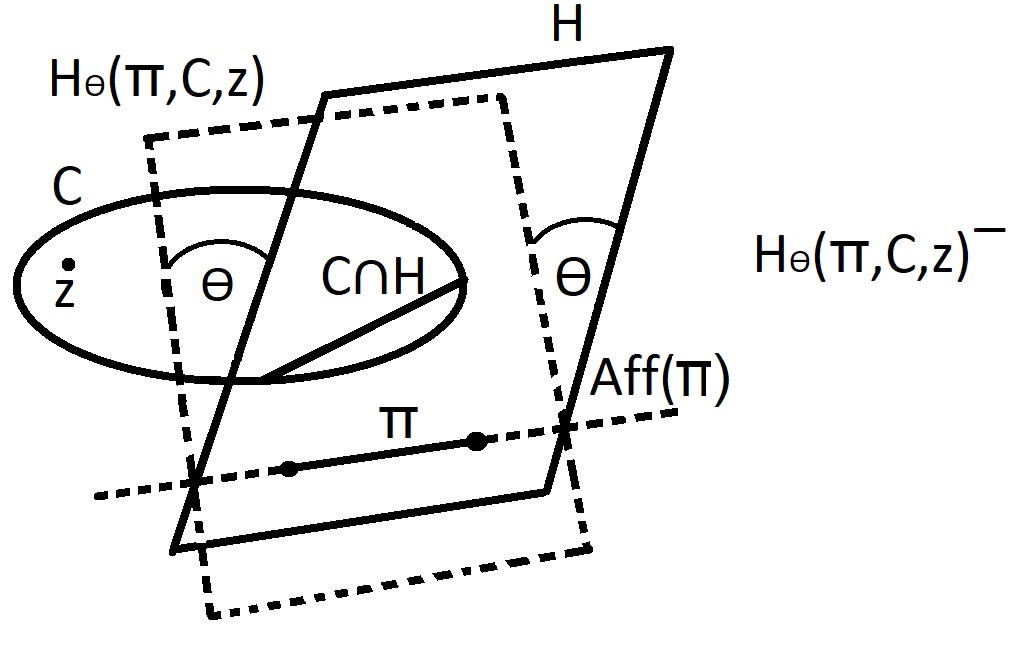}\label{rotation_fig}
\end{figure}
\end{enumerate}

\bigskip For the convenience of the reader, we will often reference to the special notation above, using \viii, \vii, \v, \vi, \ix.

\section{Outline of the proof}\label{Sketch} The proof is by induction on dimension $d$. Let $\le$ denote any of the three partial orders. In the preparatory Section \ref{v_polytopes} we show that, for dimension $d$, it is enough to have $P\le P\underset f\vee Q$ for $\dim P=\dim Q=d-1$ and $f$ a common facet.

If $Q$ is a pyramid over $f$ then polytope $P\underset f\vee Q$ is a pyramid over $P$. By inducting on the number of vertices, it is enough to show that, for a point $v$ in the half-space $\Aff_f(Q)^+$ \v\  and not in $Q$, the inequality $P\le P\underset f\vee Q$ propagates to the inequality $P\le P\underset f\vee\conv(Q,v)$. This is accomplished by gradually growing $Q$ in the sense of $\le$  towards $\conv(Q,v)$. The process of growing $Q$ breaks up into several different steps.

\medskip\noindent\emph{Theorem A.} First we show that there is an intermediate polytope $Q\subset Q_1\subset\conv(Q,v)$, such that $P\underset f\vee Q\le P\underset f\vee Q_1$ and whose facets visible from $v$ are paired with the facets of $P$ visible from $v$ so that the pairs span facets of $P\underset f\vee Q_1$. In the next step we show that for any number $\lambda<1$, sufficiently close to $1$, the homothetic transformation  -- up to a projective transformation, moving $f$ to infinity -- of the visible boundary $\partial_v(Q_1)^+$ \vii, centered at $v$ with factor $\lambda$, and $Q_1$ together span a polytope $Q_2$ with the same properties as $Q_1$ and, simultaneously, $P\underset f\vee Q_1\le P\underset f\vee Q_2$. Moreover, the process, which is summarized in Proposition \ref{Main_sequence} and represented by Figure \ref{homothety_fig}, can be iterated with respect to the same $\lambda$. This is enough for $\Infty\le$, i.e., proves Theorem A. The main vehicle for growing polytopes are the elementary moves, called the $\R$- and $\SS$-\emph{constructions}. (In Figures \ref{homothety_fig} and \ref{dim3_fig} we have not marked the edges, joining $v$ with vertices of $P$.) 

\medskip\begin{figure}[h!]
\caption{}
\vspace{.3in}
\includegraphics[%trim = 0mm 0.1in 0mm 0.1in, clip, 
scale=1.4]{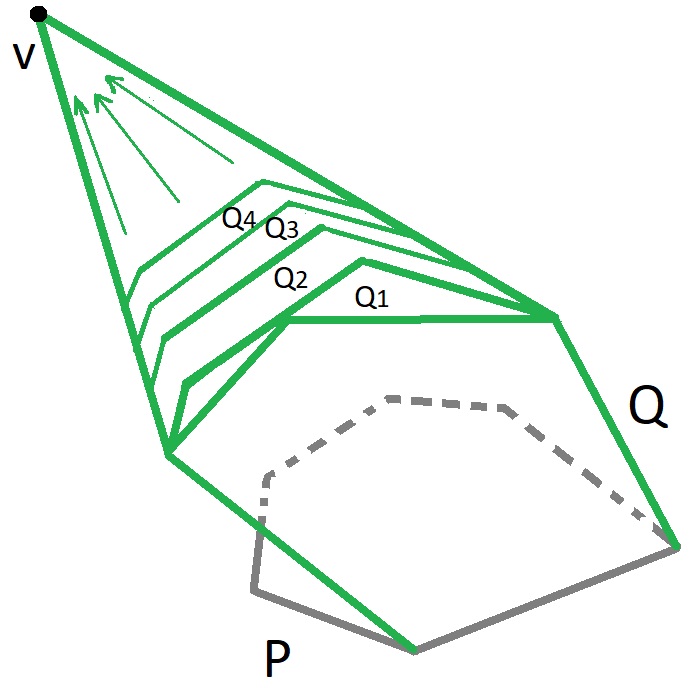}\label{homothety_fig}
\end{figure}

\bigskip\noindent\emph{Theorem B.} Since $\D\le$ is the same as $\subset$ in dimension 2, the first step for $\D\le$ as described above, reduces the problem to the inequality $P\le P\underset f\vee Q$, where $Q$ has the properties of $Q_1$ above. The main idea is to show the existence of the ascending sequence 
$$
P\underset f\vee Q\ =\ \Pi_1\ \D\le\ \Pi_2\ \D\le\ \Pi_3\ \D\le\ \ldots\ \D\le\ \Pi_n\ =\ P\underset f\vee\conv(Q,v),
$$
where $n+1$ is the number of vertices of $Q$, visible from $v$, and the facets of $\Pi_i$ in $\Aff(Q)$ grow along with the index $i$ as shown on Figure \ref{dim3_fig}:
\begin{figure}[h!]
\caption{}
\vspace{.3in}
\includegraphics[%trim = 0mm 0.1in 0mm 0.1in, clip, 
scale=1.6]{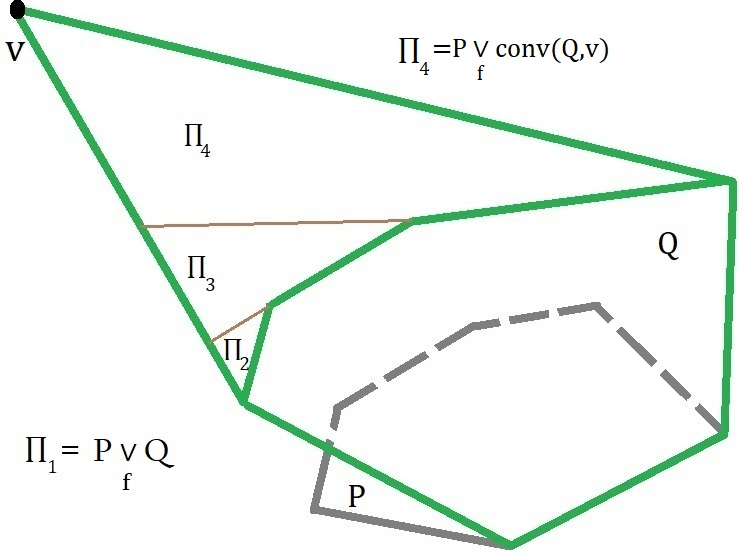}\label{dim3_fig}
\end{figure}

This is done by induction on complexity of the suspensions $\overline{\Pi_{i+1}\setminus\Pi_i}$ for $i=1,\ldots,n-1$. For instance, $\overline{\Pi_n\setminus\Pi_{n-1}}$ is of zero complexity because $\Pi_{n-1}\D\le\Pi_n$.

\bigskip\noindent\emph{Theorem C.} Using the notation in the description of the proof of Theorem A, if $n$ is sufficiently large then the polytope $Q_n$ is sufficiently close to $\conv(Q,v)$. Correspondingly, the corner cones of $P\underset f\vee Q_n$ at the vertices of $P$, visible from $v$, approximate from within the corner cones of $P\underset f\vee\conv(Q,v)$ at the same vertices. These vertices are circled in Figure \ref{induction_fig}. It can be shown that the polytope $P\underset f\vee Q_n$ can be grown along $\D\le$, staying inside $P\underset f\vee\conv(Q,v)$ and making the mentioned corner cones of $P\underset f\vee\conv(Q,v)$ exactly match with the corresponding corner cones of the new (larger) polytope $P\underset f\vee Q_n\D\le T$. This is achieved by applying the induction assumption on dimension to the cross-sections of the corner cones. In the final step, using appropriate projective and homothetic transformations, similar to the ones used in the proof of Theorem A, one chooses $n$ sufficiently large and contracts the difference between $P\underset f\vee\conv(Q,v)$ and $T$ towards the vertex $v$. Since $\q\le$ is the relaxation of $\D\le$ by discarding such small differences, Theorem C follows.  
\begin{figure}[h!]
\caption{induction}
\vspace{.15in}
\includegraphics[%trim = 0mm 0.1in 0mm 0.1in, clip, 
scale=1.6]{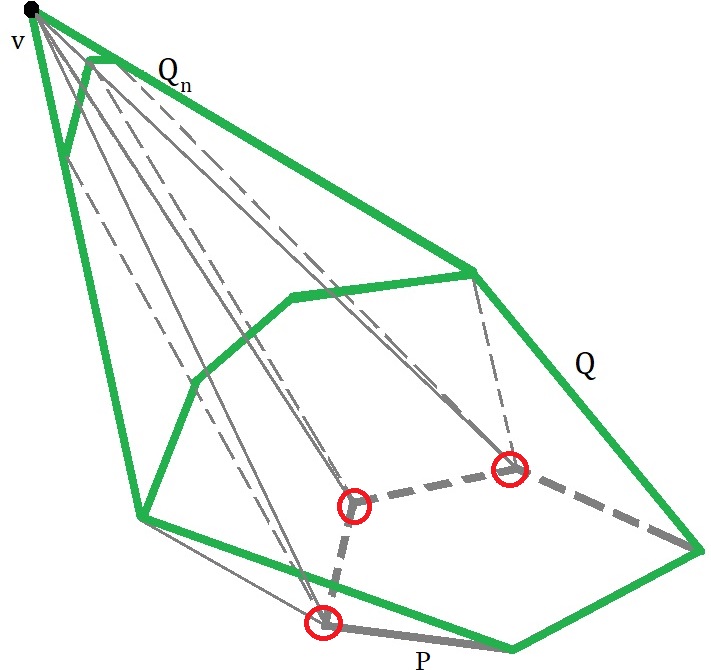}\label{induction_fig}
\end{figure}

\section{Reduction to $\vee$-polytopes}\label{v_polytopes}

In this section $\le$ denotes any of the inequalities $\D\le$, $\q\le$, and $\Infty\le$.

Consider the following statement:

\begin{equation}\label{v-assumption}
\begin{aligned}
&P\le P\underset f\vee R\ \ \text{in}\ \ \POL(d)\ \text{whenever}\ \dim R= \dim P=d-1\ \text{and}\ f=P\cap R\\ &\text{is a common facet}.\\
\end{aligned}
\end{equation}

\begin{lemma}\label{Inscribed} Assume $d\ge2$ is a natural number and \emph{(\ref{v-assumption})} holds for $d$. Assume $Q_1\subset Q_2$ are $d$-polytopes and $Q_1$ has a facet $F\subset \partial Q_2$. Then $Q_1\le Q_2$.
\end{lemma}

\begin{proof}
We will induct on the number 
$n(Q_1,Q_2):=\#\{G\in \FF(Q_1)\ |\ G\notin\partial(Q_2)\}$.

If $n(Q_1,Q_2)=0$ then $Q_1=Q_2$ and there is nothing to prove. 

Assume $Q_1\subsetneq Q_2$. Then we can find a facet $G\in\FF(Q_1)$ with $G\notin\partial(Q_2)$ and adjacent to a facet $F\in\FF(Q_1)$ with $F\subset\partial(Q_2)$. Let $g=G\cap F$. We have $\dim g=d-2$.

The space $\Aff(G)$ cuts $Q_2$ in two parts. Let $Q_2^-$ the part containing $Q_1$ and $Q_2^+$ be the other part. Let $\theta>0$ be the smallest angle for which there exists a $(d-1)$-dimensional space $H\subset\Aff(Q_2)$, satisfying the conditions:
\begin{enumerate}[\rm$\centerdot$]
\item $\Aff(G)\cap H=\Aff(g)$,
\item $Q_2^+$ is between $H$ and $\Aff(G)$,
\item The angle between $H$ and $\Aff(G)$ equals $\theta$.
\end{enumerate}
(It can happen that $H=\Aff(F)$.)

Pick any element $v$ from the nonempty set $(\vertex(Q_2)\cap H)\setminus\{g\}$ and consider the angles 
$$
0=\theta_0<\theta_1<\ldots<\theta_k=\theta,
$$
for which the hyperplane $\Aff(G)_{\theta_i}(g,Q_2^+,v)$ \ix\
 meets $\vertex(Q_2^+)$. 

For every index $i\in\{0,\ldots,k-1\}$, the piece $Q_2[\theta_i,\theta_{i+1}]$ of $Q_2^+$ between the spaces $\Aff(G)_{\theta_i}(g,Q_2^+,v)$ and $\Aff(G)_{\theta_{i+1}}(g,Q_2^+,v)$ is of the form \viii
$$
Q_2[\theta_i,\theta_{i+1}]=\big(Q_2\cap\Aff(G)_{\theta_i}(g,Q_2^+,v)\big)\underset g\vee \big(Q_2\cap\Aff(G)_{\theta_{i+1}}(g,Q_2^+,v)\big).
$$

\medskip By the assumption that (\ref{v-assumption}) holds for $d$, we have
\begin{align*}
Q_2^-\ \le\ &Q_2^-\cup Q_2[\theta_0,\theta_1]\ \le\ 
\ldots\ \le\ Q_2^-\cup Q_2[\theta_0,\theta_1]\cup\ldots\cup Q_2[\theta_{k-1},\theta_k]\ =\ Q_2.
\end{align*}

On the other hand,  $n(Q_1,Q_2^-)<n(Q_1,Q_2)$. Hence, by the induction assumption, $Q_1\le Q_2^-$.
\end{proof}

The first reduction in the proof of the main results is provided by the following

\begin{lemma}\label{v_reduction}
To prove Theorems \ref{Theorem_A}, \ref{Theorem_B}, and \ref{Theorem_C} in dimension $d\ge2$, it is enough to prove that \emph{(\ref{v-assumption})} is true for $d$.
\end{lemma}

\begin{proof}
Let $Q_1\subset Q_2$ be polytopes. By iteratively taking pyramids over $Q_1$ inside $Q_2$, we can without loss of generality assume $\dim Q_1=\dim Q_2$. Let $\{v_1,\ldots,v_n\}=\vertex(Q_2)\setminus Q_1$. Any two consecutive members in the series of inclusions
\begin{align*}
Q_1\subset\conv(Q_1,v_1)\subset\conv(Q_1,v_1,v_2)\subset\ldots\subset\conv(Q_1,v_1,\ldots,v_n)=Q_2,
\end{align*}
satisfy the condition in Lemma \ref{Inscribed}. This proves the part of Lemma \ref{v_reduction}, concerning inequalities. As for the infinitesimal nature of the pyramidal defect in Theorem \ref{Theorem_C}, it also reduces to the extensions of the form $P\subset P\underset f\vee R$ as in (\ref{v-assumption}) because the number of such extensions, involved in the proof of Lemma \ref{Inscribed}, is finite.
\end{proof}

\begin{corollary}\label{polygons}
For $d\le2$, the pyramidal growth in $\POL(d)$ is the same as the inclusion order.
\end{corollary}

\begin{proof}
The claim is obvious for $d=1$. 

When $\dim P=\dim Q=1$ and $f=P\cap Q$ is a common vertex, the polygon $P\underset f\vee Q$ is a triangle. In particular, $P\D\le P\underset f\vee Q$ and Lemma \ref{v_reduction} applies.
\end{proof}

\section{$\vee$-growth}\label{pv_growth} 

As in Section \ref{v_polytopes}, we let $\le$ denote any of the inequalities $\D\le$, $\q\le$, and $\Infty\le$.

Throughout this section, we assume $d\ge 2$ and that $P$ and $Q$ are polytopes, such that $\dim P=\dim Q=d-1$ with $f=P\cap Q$ is a common facet. 

We also fix a point $v\in\Aff_f(Q)^+\setminus Q$ \v\ and a projective transformation $\Phi$ of $\Aff(P,Q)$, moving $\Aff(f)$ to infinity, but not moving any of the points of $(P\underset f\vee Q)\setminus f$ to infinity.

Below, when we write $\partial_v(-)^+$ and $\FF_v(-)^+$ \vii, the visibility is understood with respect to the polytope $P\underset F\vee Q$.

\begin{proposition}\label{Main_sequence}
Assume $\le$ coincides with the inclusion order on $\POL(d-1)$. Then, for any number $\lambda<1$, sufficiently close to $1$, there exists an infinite sequence of polytopes
$Q=Q_0\subset Q_1\subset Q_2\subset\ldots\subset\conv(Q,v)$, such that:
\begin{enumerate}[\rm(a)]
\item For every index $i\ge1$, the sets $\FF_v(P)^+$ and $\FF_v(Q_i)^+$ are in bijective correspondence so that $\conv(D,E)\in\FF_v\big(P\underset f\vee Q_i\big)^+$  whenever $D\in\FF_v(P)^+$ and $E\in\FF_v(Q_i)^+$ correspond to each other;
\item For every index $i\ge1$, the set $\Phi(Q_i\setminus f)$ is the homothetic image of $\Phi(Q_{i-1}\setminus f)$ centered at $\Phi(v)$ with coefficient $\lambda$;
\item $P\underset f\vee Q_0\ \le P\underset f\vee Q_1\ \le\  P\underset f\vee Q_2\ \le\ \ldots$\ .
\end{enumerate}
\end{proposition}

The proof requires a preparation.

%\begin{enumerate}[\rm$\centerdot$]
%{\color{red}\item For a point $z\in\conv(Q,v)\setminus(Q\cup\{v\})$ and a facet $D\in\FF(P)$, by $\Aff(D,z)^-\subset\Aff(Q)$ we denote the affine half-space with the boundary $\Aff(D,z)\cap\Aff(Q)$ and containing $\inte(Q)$, (equivalently, containing $F$). (FIGURE)%\todo{general notation?}
%\item For an $(d-2)$-dimensional affine subspace%\todo{space \& subspace always mean affine} $s\subset\Aff(Q)$, separating $v$ and $\inte(Q)$, by $\Aff(s)^-\subset\Aff(Q)$ we denote the affine half-space, containing $\inte(Q)$.}

%{\color{blue}Observe that, for evere intermediate polytope $Q\subset Q'\subset\conv(Q,v)$, we have
%\begin{align*}
%&\partial_v(Q')^+=Q'\setminus\inte(f),\\
%&\FF_v(Q')^+=\FF(Q')\setminus\{f\}.\\
%\end{align*}}\end{enumerate}

\subsection{$\R$- and ${\SS}$-constructions}\label{R_S_constructions} We will need two auxiliary polytopal constructions.

Consider the polytope
\begin{align*}
\R(Q,P)=\conv(Q,v)\quad\bigcap\quad\bigcap_{
\tiny{\begin{aligned}
&w\in\vertex(Q)\cap\partial_v(Q)^+\\
&D\in\FF_v(P)^+\\
&\conv(D,w)\subset\partial_v\big(P\underset f\vee Q\big)^+\\
\end{aligned}}
}\Aff(D,w)_v^-.\\
&\hfill \vii\ \vi
\end{align*}
This polytope is determined by the following properties:
\begin{enumerate}[\rm$\centerdot$]
\item $Q\subset\R(Q,P)\subset\conv(Q,v)$;
\item For every point $z\in\R(Q,P)\setminus Q$, every element of $\FF_z(P\underset f\vee Q)^+$ is the convex hull of an element of $\FF_v(Q)^+$ and a face of $P$ inside $\partial_v(P)^+$ \vii, which is not a facet of $P$, i.e., the point $z$ can not see a facet of $P$;
\item $\R(Q,P)$ is the largest polytope with these properties.
\end{enumerate}

\medskip Let an intermediate polytope $Q\subset S\subset\conv(Q,v)$ satisfy the condition:
\medskip\begin{enumerate}[($\bigstar$)]
\item For every element $s\in\FF_v(S)^+$ there is an element $p\in\FF_v(P)^+$ such that $\conv(p,s)\in\FF_v(P\underset f\vee S)^+$ \vii.
\end{enumerate}

\medskip\noindent In this situation, we have the injective map
\begin{align*}
\rho:\FF_v(S)^+&\to\FF_v(P)^+,\\
s&\mapsto p.\\
\end{align*} 
Furthermore, for every element $p\in\FF_v(P)^+$, there is a unique element of $\FF_v(P\underset f\vee S)^+$, containing $p$. This assignment gives rise to a bijective  map
$$
\sigma:\FF_v(P)^+\to\FF_v(P\underset f\vee S)^+
$$
and the correspondence
$$
s\mapsto\big(\sigma(\rho(s)\big)\cap\Aff(Q)
$$
is the identity map of $\FF_v(S)^+$.

For an intermediate polytope $Q\subset S\subset\conv(Q,v)$, satisfying $(\bigstar)$, and a system of positive real numbers
$$
\Theta=\{\theta_s>0\ |\ s\in\FF_v(S)^+\},
$$
we introduce the following polytope:

\begin{align*}
{\SS}_\Theta(S,P)=\conv(Q,v)\quad\bigcap\quad\bigcap_{s\in\FF_v(S)^+}&\Aff\big((\sigma(\rho(q))\big)_{\theta_s}(\rho(s),S,v)^-\quad\bigcap\\
\bigcap_{p\in\FF_v(P)^+\setminus\Im(\rho)}&\Aff_v(\sigma(p))^-\\
&&\hfill \vi\ \ix
\end{align*}

In other words, the polytope $P\underset f\vee{\SS}_\Theta(S,P)$ is obtained from $P\underset f\vee S$ by rotating the affine hulls of the facets in $\FF_v(P\underset f\vee S)^+$, which contain the elements of $\FF_v(S)^+$ as subsets, towards $v$ by the angles $\theta_s$ about $\Aff(\rho(s))$, respectively, and not moving the affine hulls of the other facets in $\FF_v(P\underset f\vee S)^+$.

The following lemma is straightforward:

\begin{lemma}\label{R_good}
$\R(Q,P)$ satisfies $(\bigstar)$. If an intermediate polytope $Q\subset S\subset\conv(Q,v)$ satisfies $(\bigstar)$ then the polytope ${\SS}_\Theta(S,P)$ also satisfies $(\bigstar)$, where $\Theta=\{\theta_s\ |\ s\in\FF_v(S)^+\}$ and the $\theta_s>0$ are sufficiently small.
\end{lemma}

\subsection{$\R$- and ${\SS}$-growths} We need

\begin{lemma}\label{Free_zone}
For an intermediate polytope $Q\subset Q'\subset\R(P,Q)$ and a point $z\in\R(P,Q)\setminus Q'$ with $\FF_z(Q')^+=\{q'\}$, the set $\FF_z(P\underset f\vee Q')^+$ has one element. This facet of $P\underset f\vee Q'$ is the convex hull of $q'$ and a face of $P$ inside $\partial_v(P)^+$, which is not a facet of $P$. Furthermore, $P\underset f\vee Q'\le P\underset f\vee\conv(Q',z)$.
\end{lemma}

\begin{proof}
As remarked in the definition of the R-construction, $z$ can not see a facet of $P$ from outside of $P\underset f\vee Q$, and a fortiori  from outside of $P\underset f\vee Q'$. This implies
\begin{align*}
\FF_z(P\underset f\vee Q')^+=\{\conv(p',q')\}
\end{align*}
for some face $p'\subset P$, not in $\FF(P)$. As a result, the closure of the set
$$
\big(P\underset f\vee \conv(Q',z)\big)\ \setminus\ \big(P\underset f\vee Q')
$$
is the pyramid over $\conv(p',q')$ with apex at $z$. In particular,
$$
P\underset f\vee Q'\ \D\le\ P\underset f\vee \conv(Q',z).
$$
\end{proof}

\begin{lemma}\label{Q_to_R} Assume $\le$ coincides with the inclusion order on $\POL(d-1)$. Then $P\underset f\vee Q\ \le\ P\underset f\vee\R(Q,P)$.
\end{lemma}

\begin{proof}
By Lemma \ref{Free_zone}, for any intermediate polytopes $Q\subset Q'\subset Q''\subset\R(Q,P)$, forming an elementary order relation $Q'\le Q''$, we have $P\underset f\vee Q'\le  P\underset f\vee Q''$. This proves the lemma because, by the inductive assumption on the dimension, $Q\le\R(Q,P)$. 

Notice that Lemma \ref{Free_zone} applies to $\q\le$ because of the requirement $P_0\subset P'_i$, mentioned in Definition \ref{Quasi_growth}.
\end{proof}

\begin{lemma}\label{S_Theta}
Assume $\le$ coincides with the inclusion order on $\POL(d-1)$. Let an intermediate polytope $Q\subset S\subset\conv(Q,v)$ satisfy $(\bigstar)$. Then, for all sufficiently small real numbers $\theta_s>0$, $s\in\FF_v(S)^+$, we have $P\underset f\vee S\ \le\ P\ \underset f\vee\ {\SS}_\Theta(S,P)$, where $\Theta=\{\theta_s\ |\ s\in\FF_v(S)^+\}$.
\end{lemma}

\begin{proof}
Assume $\FF_v(S)^+=\{s_1,\ldots,s_n\}$. For every index $i\in\{1,\ldots,n\}$, let $\gamma_i$ be the barycenter of $s_i$.\footnote{For the present argument, we can use any points in $\inte(s_i)$; the barycenters are adjusted to working with $\POL_k(d)$ for a subfield $k\subset\RR$.} Choose $z_i$ in $[\gamma(s_i),v]$, sufficiently close to but different from $\gamma_i$, so that the following two conditions are satisfied:

\begin{align*}
s_i\in\R\big(\conv(S,s_1,\ldots,s_n),P\big)\quad\text{and}\quad[z_i,z_j]&\cap\inte(S)\not=\emptyset,\\
&1\le i, j\le n,\ \ i\not=j.
\end{align*}

\medskip\noindent Then Lemma \ref{Free_zone} implies

{\small
$$
P\ \underset f\vee\ S\ \le\ P\ \underset f\vee\ \conv(S,s_1)\ \le\ P\ \underset f\vee\ \conv(S,s_1,s_2)\le\ \ldots\  \le\ P\ \underset f\vee\ \conv(S,s_1,\ldots,s_n).
$$
}

By Lemma \ref{Q_to_R}, we also have
$$
P\ \underset f\vee\ \conv(S,s_1,\ldots,s_n)\ \le\ P\ \underset f\vee\ \R\big(\conv(S,s_1,\ldots,s_n),P\big).
$$
We are done because 
$$
P\ \underset f\vee\ \R\big(\conv(S,s_1,\ldots,s_n),P\big)={\SS}_\Theta(S,P)
$$
for the appropriate $\Theta=\{\theta_s\ |\ s\in\FF_v(S)^+\}$. Namely, the angles $\theta_s$ are determined by the condition
$$
z_s\in\Aff(\sigma(\rho(s))_{\theta_s}(\rho(s),S,v),\qquad s\in\FF_v(S)^+.
$$ 
\end{proof}

\begin{lemma}\label{Q1}
Assume $\le$ coincides with the inclusion order on $\POL(d-1)$. Then there is an intermediate polytope $\R(Q,P)\subset Q_1\subset\conv(Q,v)$, such that:
\begin{enumerate}[\rm(a)]
\item
$Q_1$ satisfies $(\bigstar)$;
\item The map $\rho:\FF_v(Q_1)^+\to\FF_v(P)^+$ from Section \ref{R_S_constructions} is a bijection;
\item $P\underset f\vee\R(Q,P)\ \le\ P\underset f\vee Q_1$.
\end{enumerate}
\end{lemma}

\begin{proof}
To simplify notation, put $R:=\R(Q,P)$. By Lemma \ref{R_good}, the polytope $R$ satisfies $(\bigstar)$.  We have the maps, mentioned in Section \ref{R_S_constructions}:
\begin{align*}
&\rho:\FF_v(R)^+\to\FF_v(P)^+,\\
&\sigma:\FF(P)^+\to\FF_v(P\underset f\vee R)^+.
\end{align*}

Assume the injective map $\rho$ is not bijective. We will promote $\rho$ to a bijection by inductively constructing a sequence of polytopes 
$$
R=R_0\subset R_1\subset\ldots\subset R_n\subset \conv(Q,v),
$$
such that:

\begin{enumerate}[\rm$\centerdot$]
\item Each $R_i$ satisfies $(\bigstar)$;
\item
$P\underset f\vee R_0\ \le\ P\underset f\vee R_1\ \le\ \ldots\ \le\  P\underset f\vee R_n$;
\item $\#\FF(R_0)<\#\FF(R_1)<\ldots<\#\FF(R_n)=\#\FF(P)$.
\end{enumerate}

We will use the following notation for the corresponding maps:
\begin{align*}
&\rho_i:\FF_v(R_i)^+\to\FF_v(P)^+,\\
&\sigma_i:\FF(P)^+\to\FF_v(P\underset f\vee R_i)^+.
\end{align*}
In particular, $\rho_0=\rho$ and $\sigma_0=\sigma$.

Assume, after $t$ steps, we have produced polytopes $R_1,\ldots,R_t$ with the desired properties. Assume $\#\FF(R_t)<\#\FF(P)$ or, equivalently, $\rho_t$ is not a bijection. Choose sufficiently small real numbers 
$$
\theta_r>0,\qquad r\in\FF_v(R_t)^+,
$$
and put $\Theta=\{\theta_r\}$. The polytope $R_{t+1}$ satisfies $(\bigstar)$ (Lemma \ref{R_good}). By Lemma \ref{S_Theta}, we have $P\underset f\vee R_t\ \le\ P\underset f\vee R_{t+1}$. Simultaneously, 
\begin{align*}
\{p\in\FF_v(P)^+\ &|\ \sigma_k(p)\cap R_t\ \text{is a facet in}\ 
\partial_v(R_t)^+\}\subsetneq\\
&\{p\in\FF_v(P)^+\ |\ \sigma_{t+1}(p)\cap R_{t+1}\ \text{is a facet in}\ 
\partial_v(R_{t+1})^+\}.
\end{align*}
or, equivalently, $\Im(\rho_t)\subsetneq\Im(\rho_{t+1})$. 
\end{proof}

\subsection{Projective transformation $\Psi_\lambda$}\label{Psi_lambda} To complete the proof of Proposition \ref{Main_sequence}, we need another projective transformation.

Let $A$ be a space and $H\not=H'$ be parallel codimension one subspaces of $A$. Choose a point $w\in H$. Then, for any real number $\lambda\not=0$, there is a (unique) projective transformation $\Psi_\lambda$ of $A$, such that the restriction $(\Psi_\lambda)|_H$ is the homothety, centered at $w$ with coefficient $\lambda$, and the restriction $(\Psi_\lambda)|_{H'}$ is the identity map. 

Next we give a geometric description of $\Psi_\lambda$ when $0<\lambda<1$. This description implies a property of $\Psi_\lambda$, which will be important in the second use of $\Psi_\lambda$ in Section \ref{Quasi_proof}. Let $\curvearrowright$ be any $90^\circ$-rotation of $A$ inside $\oplus_\NN\RR$ about $H'$ and $\curvearrowleft$ be the $90^\circ$-rotation in the opposite direction. Denote by $\overset\curvearrowright A$, $\overset\curvearrowright H$, and $\overset\curvearrowright w$ the corresponding images. Consider the point $\pi\in\Aff(w,\overset\curvearrowright w)$, such that $\overset\curvearrowright w$ is between $\pi$ and $w$ and $\frac{\|\pi-\overset\curvearrowright w\|}{\|\pi-w\|}=\lambda$; see Figure \ref{projective_fig} below.%\todo{fixed}:
\begin{figure}[h!]
\caption{}
\vspace{.15in}
\includegraphics[%trim = 0mm 0.1in 0mm 0.1in, clip, 
scale=1.4]{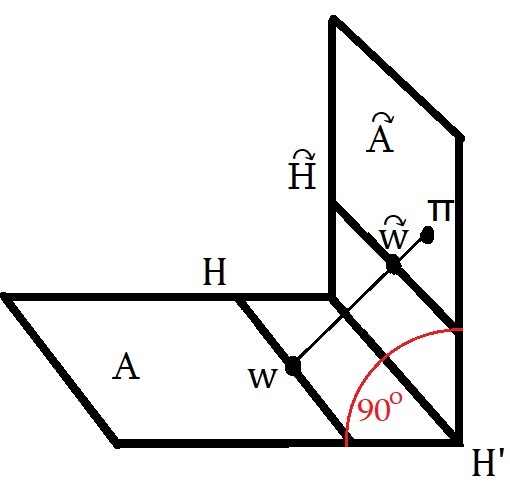}\label{projective_fig}
\end{figure}

Consider the polar projection $\text{proj}_\pi:A\dasharrow\overset\curvearrowright A$ from the pole $\pi$. Its domain includes the half-space of $A$, bounded by $H'$ and containing $H$. We have
$$
\Psi_\lambda=\curvearrowleft\ \circ\ \text{proj}_\pi.
$$
When $\lambda$ converges to $0$ from the right and $A$, $H$, $H'$, $w$ stay fixed, the pole $\pi$ converges to $\overset\curvearrowright w$ by sliding along the line $\Aff(w,\overset\curvearrowright w)$. In view of the equality above we arrive at the conclusion that, for $A$, $H$, $H'$, $w$ as above and a point in the half-space of $A$, bounded by $H'$ and containing $H$%\todo{fixed},
\begin{equation}\label{Polar_contraction}
\lim_{\lambda\to0^+}\Psi_\lambda(z)=w.
\end{equation}

\subsection{Proof of Proposition \ref{Main_sequence}}\label{Proof_prop}

Let $Q_1$ be as in Lemma \ref{Q1}.

First we observe that, for any number $0<\lambda<1$, there is a sequence of polytopes 
$$
Q_1\subset Q_2\subset\ldots\subset\conv(Q,v)
$$
and a system of real numbers
$$
\theta_{iq}>0,\quad q\in\FF_v(Q_{i-1})^+,\quad i=2,3,\ldots\ ,
$$
such that, for every $i\ge2$, we have:

\begin{enumerate}[$\centerdot$]
\item $Q_i={\SS}_{\Theta_i}(Q_{i-1},P)$ with $\Theta_i=\{\theta_{iq}\}$;
\item
$\Phi(Q_i\setminus f)$ is the homothetic image of $\Phi(Q_{i-1}\setminus f)$ with center $\Phi(v)$ and factor $\lambda$.
\end{enumerate}

In fact, for the existence of the family 
$\{\theta_{2q}>0\ |\ q\in\FF_v(Q_1)^+\}$, such that $Q_1$ and the corresponding polytope $Q_2$ have the desired properties, one only needs to adjust the angles $\theta_{2q}$ to fit $\partial_v(Q_2)^+$ into the homothety condition. Once we determine the $\theta_{2q}$, the existence of the next family of real numbers $\{\theta_{3q}>0\ |\ q\in\FF_v(Q_2)^+\}$ with the similar properties is obvious for the same reason, and the process can be iterated.

By Lemma \ref{S_Theta}, if $\lambda<1$ is sufficiently close to $1$ then $P\underset f\vee Q_1\le P\underset f\vee Q_2$. We want to show that, for the \emph{same} $\lambda$, we have

$$
P\underset f\vee Q_{i-1}\ \le\ P\underset f\vee Q_i,\qquad i=2,3,\ldots
$$

This is done as follows. Let $\Psi_\lambda$ be the projective transformation, introduced in Section \ref{Psi_lambda}, with the following specializations: 
\begin{align*}
&A=\Aff\big(\Phi\big((P\underset f\vee Q)\setminus f\big)\big),\\
&H=\Aff\big(\Phi(Q\setminus f)\big),\\
&H'=\Aff\big(\Phi(P\setminus f)\big),\\
&w=\Phi(v).
\end{align*}
We have 

$$
Q_i=\conv\big(Q,\ (\Phi^{-1}\circ\Psi\circ\Phi)(Q_{i-1}\setminus f)\big).
$$
In particular, by pushing forward the chain of elementary order relations along the projective transformation $\Phi^{-1}\circ\Psi\circ\Phi$, we arrive at the implication
$$
P\underset f\vee Q_{i-1}\ \le\ P\underset f\vee Q_i\quad \Longrightarrow\quad P\underset f\vee Q_i\ \le\ P\underset f\vee Q_{i+1}.
$$
Consequently, the inequality $P\underset f\vee Q_{i-1}\ \le\ P\underset f\vee Q_i$ propagates from the initial value $i=2$ to all values of $i$. 
\qed

\subsection{Transfinite pyramidal growth}\label{4.5} Here we prove Theorem \ref{Theorem_A}. We will induct on dimension, the base one-dimensional case being obvious. Assume the claim has been shown for $\POL(d-1)$. By Lemmas \ref{v_reduction}, it is enough to show (\ref{v-assumption}) holds for $d$, i.e., we have to show $P\Infty\le P\underset f\vee Q$.

Assume 
$\{v_0,v_1,\ldots,v_n\}=\vertex(Q)\setminus f$
and consider the polytopes
\begin{align*}
T_i=\conv(f,v_0,\ldots,v_i),\qquad i=0,\ldots,n.
\end{align*}
In view of Proposition \ref{Main_sequence}, we can write
\begin{align*}
P\ \D\le\ \conv(P,v_0)=P\underset f\vee T_0\ \Infty\le\ P\underset f\vee T_1\ \Infty\le\ \ldots\ \Infty\le\ P\underset f\vee T_n\ =\ P\underset f\vee Q.\qed
\end{align*}

\section{Pyramidal growth of 3-polytopes}\label{Dim_three} 

By Corollary \ref{polygons}, we already know that the pyramidal growth is the inclusion order in $\POL(d)$, $d\le2$.

Here we prove Theorem \ref{Theorem_B}, with the use of an essential part of Proposition \ref{Main_sequence}.

By Lemma \ref{v_reduction}, to prove Theorem \ref{Theorem_B}, it is enough to show $P'\D\le P'\underset{f'}\vee Q'$, where $P'$ and $Q'$ are two polygons and $f'=P'\cap Q'$ is a common edge. As we did in Section \ref{4.5}, assume 
$\{v_0,v_1,\ldots,v_n\}=\vertex(Q')\setminus f'$
and consider the polygons
\begin{align*}
T_i=\conv(f,v_0,\ldots,v_i),\qquad i=0,\ldots,n.
\end{align*}
Then $P'\D\le\conv(P',v_0)$ and, moreover, every inclusion in the sequence  
$$
\conv(P',v_0)=P'\underset{f'}\vee T_0\ \subset\ P'\underset{f'}\vee T_1\ \subset\ \ldots\ \subset\ P'\underset{f'}\vee T_n\ =\ P'\underset{f'}\vee Q'.
$$
can be changed to $\D\le$ if we show that
\begin{equation}\label{3_step1}
P\underset f\vee Q\ \D\le\ P\underset f\vee\conv(Q,v),
\end{equation}
whenever:
\begin{enumerate}[\rm$\centerdot$]
\item
$\dim P=\dim Q=2$, 
\item $f=P\cap Q$ is a common edge,
\item $v\in\Aff_f(Q)^+\setminus Q$.\ \v
\end{enumerate}

We will induct on $\#\FF_v(P)^+$ \vii, where the visibility is with respect to $P\underset f\vee Q$.

In the base case, when $\#\FF_v(P)^+=1$, we pick an intermediate polygon $\R(P,Q)\subset Q_1\subset\conv(Q,v)$ as in Lemma \ref{Q1}, and write

$$
P\underset f\vee Q\ \D\le\ P\underset f\vee\R(P,Q)\ \D\le\ P\underset f\vee Q_1\ \D\subset\ P\underset f\vee\conv(Q,v),
$$
where the leftmost inequality is due to Lemma \ref{Q_to_R}.

Assume $n>1$ and we have shown (\ref{3_step1}) for $\#\FF_v(P)^+\le n-1$. 

Consider the case $\#\FF_v(P)^+=n$. In view of Lemmas \ref{Q_to_R} and \ref{Q1}, we can assume that there is bijective correpondence between $\FF_v(Q)^+$ and $\FF_v(P)^+$ so that the facets in $\FF_v(P\underset f\vee Q)^+$ are of the form $\conv(p,q)$, where the edges $p\in\FF_v(P)^+$ and $q\in\FF_v(Q)^+$ correspond to each other. 

By successively enumerating the adjacent vertices in $P$ and $Q$, visible from $v$, we can assume:

\begin{align*}
&\{x_i\}_{i=1}^{n+1}=\partial_v(P)^+\cap\vertex(P),\\
&\{y_i\}_{i=1}^{n+1}=\partial_v(Q)^+\cap\vertex(Q),\\
&\big\{[x_i,x_{i+1}]\big\}_{i=1}^n=\FF_v(P)^+,\\
&\big\{[y_i,y_{i+1}]\big\}_{i=1}^n=\FF_v(P)^+,\\
&\big\{\conv(x_i,x_{i+1},y_i,y_{i+1})\big\}_{i=1}^n=\FF_v(P\underset f\vee Q)^+.\ \vii
\end{align*}

Consider the family of polytopes $\Pi_i$ and polygons $P_i$ and $Q_i$:

\begin{align*}
&\Pi_i=\big(P\underset f\vee\conv(Q,v)\big)\ \bigcap\ \bigcap_{j=i}^n \Aff\big(x_i,x_{i+1},y_i,y_{i+1}\big)_v^-,\  \vi\\
&P_i=\Pi_i\cap\Aff(x_i,x_{i+1},y_i,y_{i+1}),\\
&Q_i=\big(P_i\setminus\conv(x_i,x_{i+1},y_i,y_{i+1})\big)\cup[x_i,y_i],\\
&\qquad\qquad\qquad\qquad\qquad\qquad\qquad\qquad\qquad i=1,\ldots,n.
\end{align*}
 
We have:
\begin{align*}
&\Pi_1=P\underset f\vee Q,\\
&\Pi_n\ \D\subset\ P\underset f\vee\conv(Q,v),\\
&\overline{\Pi_{i+1}\setminus\Pi_i}\ =\ P_i\underset{[x_{i+1},y_{i+1}]}\vee Q_{i+1}\ \ \text{for}\ \ i=1,\ldots,n-1.
\end{align*}
%(Notice, the equalities for $\overline{\Pi_{i+1}\setminus\Pi_i}$ are essentially three-dimensional phenomena.)

Consequently, to prove (\ref{3_step1}) it is enough to show
\begin{equation}\label{3_step2}
P_i\ \ \D\le\ \ P_i\underset{[x_{i+1},y_{i+1}]}\vee Q_{i+1},\qquad i=1,\ldots,n-1.
\end{equation}

For every index $1\le i\le n$, the vertices of $P_i$ are determined as follows:
\begin{align*}
\vertex(P_i)=\{x'_1,x'_2,\ldots,x'_{i-1},x_i,x_{i+1},y_{i+1},y'_1\},
\end{align*}
where
\begin{align*}
x'_1&=[x_1,v]\cap\Aff(x_i,x_{i+1},y_i,y_{i+1}),\\
x'_2&=[x_2,v]\cap\Aff(x_i,x_{i+1},y_i,y_{i+1}),\\
&\ldots\ldots\ldots\\
x'_{i-1}&=[x_{i-1},v]\cap\Aff(x_i,x_{i+1},y_i,y_{i+1}),\\
y'_1&=[y_1,v]\cap\Aff(x_i,x_{i+1},y_i,y_{i+1}).
\end{align*}

In particular, $\#\vertex(P_i)=i+3$. Therefore, (\ref{3_step2}) follows from the following

\begin{lemma}
Assume $n\ge2$ and (\ref{3_step1}) holds for $\#\FF_v(P)^+\le n-1$. If $P'$ and $Q'$ are polygons, $f'=P'\cap Q'$ is a common edge, and $\#\vertex(P')\le n+2$ then $P'\D\le P'\underset{f'}\vee Q'$. 
\end{lemma}

\begin{proof}
Let $z_1,\ldots,z_k$ be the vertices of $P'$, enumerated in the cyclic order so that $f'=[z_1,z_k]$. There exists a plane $H\subset\Aff(P',Q')$, such that
\begin{align*}
\emptyset\not=Q'\cap H\subset\partial(Q')\quad\text{and}\quad P'\cap H=
\begin{cases}
z_{l+1},\ \ \text{if}\ \ k=2l+1,\\
\big[z_l,z_{l+1}\big],\ \ \text{if}\ \ k=2l.
\end{cases}
\end{align*}
To see this, pick a line $L\subset\Aff(P')$ with
$$
P'\cap L=
\begin{cases}
z_{l+1},\ \ \text{if}\ \ k=2l+1,\\
\big[z_l,z_{l+1}\big],\ \ \text{if}\ \ k=2l,
\end{cases}
$$
and start rotating the plane $\Aff(P')$ about $L$ away from $P'\underset{f`}\vee Q'$, until it hits $\partial Q'$ from the other side.

The plane $H$ contains an element $v'\in\vertex(Q')$. Consider the pyramid 
$$
\conv(P',v')=P'\underset{f'}\vee\conv(f',v').
$$ 
It splits up the set
$$
\big(P'\underset{f'}\vee Q'\big)\ \setminus\ \big(P'\underset{f'}\vee\conv(f',v')\big)
$$
in such a way that, for every vertex $w'\in\vertex(Q')\setminus\{z_1,z_k,v'\}$, we have

\begin{align*}
\FF_{w'}(P')^+\subset
\begin{cases}
\{[z_1,z_2],\ldots,[z_l,z_{l+1}]\}\ \text{or}\ \{[z_{l+1},z_{l+2}],\ldots[z_{k-1},z_k]\},\ \text{if}\quad k=2l+1,\\
\{[z_1,z_2],\ldots,[z_{l-1},z_l]\}\ \text{or}\ \{[z_{l+1},z_{l+2}],\ldots[z_{k-1},z_k]\},\ \text{if}\quad k=2l,\\
\end{cases}
\end{align*}

\medskip\noindent where the visibility is understood with respect to $P'\underset{f'}\vee\conv(f',v')$. In particular, for every vertex $w'\in\vertex(Q')\setminus\{z_1,z_k,v'\}$, we can write
\begin{align*}
\#\FF_{w'}(P')^+\le
\begin{cases}
l\le\max(1,k-3)\le n-1,\ \ \text{if}\ \ k=2l+1\ \ \text{and}\ \ l\ge1,\\
l-1\le k-3\le n-1,\ \ \text{if}\ \ k=2l\ \ \text{and}\ \ l\ge2.\\
\end{cases}
\end{align*}

\medskip\noindent Since the visibility of facets and vertices of $P'$ does not improve when one passes from $P'\underset{f'}\vee\conv(f',v')$ to $P'\underset{f'}\vee Q''$ for any intermediate polytope $\conv(f',v')\subset Q''\subset Q'$, the induction assumption yields the sequence of inequalities

\begin{align*}
P'\ \D\subset\ \conv(P',v')\ =\ P'&\underset{f'}\vee\conv(f',v')\ \D\le\ P'\underset{f'}\vee\conv(f',v',w'_1)\ \D\le\ \ldots\\
&\D\le\ P'\underset{f'}\vee\conv(f',v',w'_1,\ldots,w'_{k-3})\ =\ P'\underset{f'}\vee Q',
\end{align*}

\medskip\noindent where $\{w'_1,\ldots,w'_k\}=\vertex(Q')\setminus\{z_1,z_k,v'\}$.
\end{proof}

\section{Quasi-pyramidal growth}\label{Quasi_proof} 

Here we prove Theorem \ref{Theorem_C}, with a crucial use of Proposition \ref{Main_sequence}. 

\begin{lemma}\label{Blowing_corner}
Let $R\subset S$ be two $(d+1)$-polytopes, sharing a vertex $w$. Then there exists a polytope $T$, for which $R\D\le T\subset S$ and the corner cones of $T$ and $S$ at $w$ coincide, i.e., $\RR_+(T-w)+w=\RR_+(S-w)+w$.
\end{lemma}

\begin{proof}
We can assume $w=0$. Fix a $d$-space $H\subset(\RR S)\setminus\{0\}$, traversing the cone $\RR_+S$ \cite[Prop.~1.21]{Kripo}. By the assumption, there exist pyramidal extensions
\begin{align*}
\RR_+R\cap H\ =\  \Sigma_0\ \D\subset\ \Sigma_1\ \D\subset\ \ldots\ \D\subset\ \Sigma_m\ =\ \RR_+S\cap H.
\end{align*}
Assume $\zeta_j=\vertex(\Sigma_j)\setminus\Sigma_{j-1}$ for $j=1,\ldots,m$. For a system of real numbers
$$
0<\mu_m\ll\mu_{n-1}\ll\ldots\ll\mu_1\ll1,
$$
where $a\ll b$ means ``$\frac ab$ is sufficiently small'', the points $z_j=\mu_j\zeta_j$ satisfy the conditions:
\begin{align*}
&z_1,\ldots,z_m\in S,\\
&\vertex(\conv(R,z_1,\ldots,z_j))=\vertex(R)\cup\{z_1,\ldots,z_j\},\\
&R\ \D\subset\ \conv(R,z_1)\ \D\subset\ \conv(R,z_1,z_2)\ \D\subset\ \ldots\ \D\subset\ \conv(R,z_1,\ldots,z_m),\\
&\RR_+\conv(R,z_1,\ldots,z_m)=\RR_+S.
\end{align*}

\medskip\noindent We can choose $T=\conv(R,z_1,\ldots,z_m)$.
\end{proof}

Assume $\D\le$ is the inclusion order on $\POL(d)$. By Lemma \ref{v_reduction}, it is enough to show that the conclusion of Theorem \ref{Theorem_C} holds for a pair of polytopes $P\subset P\underset f\vee Q'$, where $P$ and $Q'$ are $d$-polytopes with $f=P\cap Q'$ a common facet. 

Assume $\vertex(Q')\setminus f=\{v_1,\ldots,v_n\}$. We will induct on $n$.

For $n=1$, we are done because $P\D\le\ \conv(P,v_1)\ =\ P\underset f\vee\conv(f,v_1)$.

Put $Q=\conv(Q',v_1,\ldots,v_{n-1})$ and $v=v_n$. Let a number $\lambda$, a projective transformation $\Phi$, and a sequence of polytopes $Q=Q_0\subset Q_1\subset Q_2\subset\ldots$ be as in Proposition \ref{Main_sequence}. Assume $\Psi_\lambda$ is the projective transformation, used in Section \ref{Proof_prop}, i.e.,  $\Psi_\lambda$ is the homothetic contraction of $\Aff(\Phi(Q\setminus f))$, centered at $\Phi(v)$ with coefficient $\lambda$, and the identity map on $\Aff(\Phi(P \setminus f))$.

We can assume that the polytopes $P\underset f\vee Q_1$ and $P\underset f\vee\conv(Q,v)$ have the same corner cones at every vertex from $\vertex(f)$: this can be achieved by starting with $Q_2$ instead of $Q_1$, if necessary. By Lemma \ref{Blowing_corner}, iteratively applied to the elements of $\vertex(P)\setminus f$, viewed as vertices in $\vertex(P\underset f\vee Q_1)\setminus f$, we find a polytope $T$ with the properties:
\begin{enumerate}[\rm(i)]
\item The corner cones of $T$ and $P\underset f\vee\conv(Q,v)$ at every vertex of $P$ coincide;
\item
$P\underset f\vee Q_1\ \D\le\ T\ \subset P\underset f\vee\conv(Q,v)$.
\end{enumerate}

In view of the convergence (\ref{Polar_contraction}) in Section \ref{Psi_lambda} and the property (i) above, for any real number $\epsilon>0$, the following inclusion is satisfied if $k\in\NN$ is sufficiently large:

\begin{align*}
\big(\Phi^{-1}\circ\Psi_{\lambda^k}\circ\Phi\big)(u)\in\B_\epsilon(v),\qquad u\in\vertex(T)\setminus P,%(P\underset f\vee Q),
\end{align*}
where $\B_\epsilon(v)$ stands for the $\epsilon$-ball, centered at $v$. In particular, for any $\epsilon>0$ and sufficiently large $k\in\NN$, we have

\begin{align*}
\big(P\underset f\vee\conv(Q,v)\big)\setminus\B_\epsilon(v)\ =\ \conv\big(Q,\ (\Phi\circ\Psi_{\lambda^k}\circ\Phi)(T\setminus f)\big)\setminus\B_\epsilon(v).
\end{align*}
Consequently, for any $\epsilon>0$ and sufficiently large $k\in\NN$ (depending on $\epsilon$), we can find a pyramidal extension
$$
T'\ \D\subset\ P\underset f\vee\conv(Q,v),
$$
with the properties:

\begin{equation}\label{Last_pyramidal}
\begin{aligned}
T'\ \subset\ \conv\big(Q,\ (\Phi\circ\Psi_{\lambda^k}\circ\Phi)(T\setminus f)\big)\quad\text{and}\quad\d\big(T',P\underset f\vee\conv(Q,v)\big)<\epsilon.
\end{aligned}
\end{equation}

We write
\begin{align*}
P\underset f\vee Q_{k+1}\ =\ &\conv\big(Q,\ \big(\Phi^{-1}\circ\Psi_{\lambda^k}\circ\Phi\big)\big((P\underset f\vee Q_1)\setminus f\big)\big)\ \D\le\\ 
&\conv\big(Q,\ \big(\Phi\circ\Psi_{\lambda^k}\circ\Phi\big)(T\setminus f)\big)\ \q\le\ P\underset f\vee\conv(Q,v),
\end{align*}
where:
\begin{enumerate}[\rm$\centerdot$]
\item
the inequality $\D\le$ results from pushing forward the inequality $P\underset f\vee Q_1\D\le T$ in (ii) above along the composite transformation $\Phi^{-1}\circ\Psi_{\lambda^k}\circ\Phi$,
\item the inequality $\q\le$ follows for the inclusion in (\ref{Last_pyramidal}) for $\epsilon$ sufficiently small. 
\end{enumerate}

\medskip We also have $P\underset f\vee Q\ =\ P\underset f\vee Q_0 \ \D\le\ P\underset f\vee Q_{k+1}$, whereas the inequality in (\ref{Last_pyramidal}) implies
$$
\delta\big(P,P\underset f\vee\conv(Q,v)\big)\le\delta\big(P,P\underset f\vee Q\big)+\epsilon.
$$
\qed

\bigskip\begin{remark}\label{Almost_quasi}
The proof of Theorem \ref{Theorem_C} shows that, for the coincidence of $\q\le$ and $\subset$ on $\POL(d+1)$, it is enough to have the following `quasi' version of Lemma \ref{Blowing_corner}: if $\q\le$ and $\subset$ coincide on $\POL(d)$ and $R\subset S$ are two $(d+1)$-polytopes, sharing $0$ as a vertex, then there exists an intermediate polytope $R\q\le T\subset S$, such that $\RR_+ T=\RR_+S$. Unlike Proposition \ref{Main_sequence}, whose proof also uses induction on dimension and goes through for $\q\le$, we do not know how to involve the quasi-pyramidal growth in the context of Lemma \ref{Blowing_corner}. As for the infinitesimal nature of the pyramidal defect, it is easy to show that it causes no additional challenge.
\end{remark}

\bibliographystyle{plain}
\bibliography{bibliography}

\begin{thebibliography}{10}

\bibitem{Paco}
M\'{o}nica Blanco and Francisco Santos.
\newblock Enumeration of lattice 3-polytopes by their number of lattice points.
\newblock {\em Discrete Comput. Geom.}, 60(3):756--800, 2018.

\bibitem{Brondsted}
Arne Br{\tiny\O}ndsted.
\newblock {\em An introduction to convex polytopes}, volume~90 of {\em Graduate
  Texts in Mathematics}.
\newblock Springer-Verlag, New York-Berlin, 1983.

\bibitem{Unico}
Winfried Bruns and Joseph Gubeladze.
\newblock Normality and covering properties of affine semigroups.
\newblock {\em J. Reine Angew. Math.}, 510:161--178, 1999.

\bibitem{Kripo}
Winfried Bruns and Joseph Gubeladze.
\newblock {\em Polytopes, rings, and $K$-theory}.
\newblock Springer Monographs in Mathematics. Springer-Verlag, New York, 2009.

\bibitem{Caratex}
Winfried Bruns, Joseph Gubeladze, Martin Henk, Alexander Martin, and Robert
  Weismantel.
\newblock A counterexample to an integer analogue of {C}arath\'{e}odory's
  theorem.
\newblock {\em J. Reine Angew. Math.}, 510:179--185, 1999.

\bibitem{Quantum}
Winfried Bruns, Joseph Gubeladze, and Mateusz Micha{\l}ek.
\newblock Quantum jumps of normal polytopes.
\newblock {\em Discrete Comput. Geom.}, 56(1):181--215, 2016.

\bibitem{Elementary_moves}
Julien David, Lionel Pournin, and Rado Rakotonarivo.
\newblock Elementary moves on lattice polytopes.
\newblock {\em J. Combin. Theory Ser. A}, 172:105200, 2020.

\bibitem{Gruber}
Peter~M. Gruber.
\newblock Aspects of approximation of convex bodies.
\newblock In {\em Handbook of convex geometry, {V}ol. {A}, {B}}, pages
  319--345. North-Holland, Amsterdam, 1993.

\bibitem{Grunbaum}
Branko Gr\"{u}nbaum.
\newblock {\em Convex polytopes}, volume 221 of {\em Graduate Texts in
  Mathematics}.
\newblock Springer-Verlag, New York, second edition, 2003.

\bibitem{Anderson}
Joseph Gubeladze.
\newblock The {A}nderson conjecture and a maximal class of monoids over which
  projective modules are free.
\newblock {\em Mat. Sb. (N.S.)}, 135(177)(2):169--185, 271, 1988.

\bibitem{Nilpotence}
Joseph Gubeladze.
\newblock The nilpotence conjecture in {$K$}-theory of toric varieties.
\newblock {\em Invent. Math.}, 160(1):173--216, 2005.

\bibitem{Elrows}
Joseph Gubeladze.
\newblock Unimodular rows over monoid rings.
\newblock {\em Adv. Math.}, 337:193--215, 2018.

\bibitem{Functors}
Joseph Gubeladze.
\newblock Affine-compact functors.
\newblock {\em Adv. Geom.}, 19(4):487--504, 2019.

\bibitem{Cones}
Joseph Gubeladze and Mateusz {Micha{\l}ek}.
\newblock The poset of rational cones.
\newblock {\em Pacific J. Math.}, 292(1):103--115, 2018.

\end{thebibliography}

\end{document}